\documentclass{article}

\usepackage{amsmath,amsthm,amssymb,amsfonts}
\usepackage{verbatim}

\usepackage{stmaryrd} 
\usepackage{hyperref}

\usepackage[colorinlistoftodos]{todonotes}

\usepackage{tikz}
\usepackage{tkz-berge, tkz-graph}
\tikzset{vertex/.style={circle,draw,fill,inner sep=0pt,minimum size=1mm}}

\newcommand{\yzgrid}[5] 
{
  \foreach \y in {#2,...,#3} {
      \draw[white,line width=3pt] (\y,#4,#1) -- (\y,#5,#1);
  }
  \foreach \z in {#4,...,#5} {
      \draw[white,line width=3pt] (#2,\z,#1) -- (#3,\z,#1);
      }
  \foreach \y in {#2,...,#3} {
      \draw[densely dotted] (\y,#4,#1) -- (\y,#5,#1);
  }
  \foreach \z in {#4,...,#5} {
      \draw[densely dotted] (#2,\z,#1) -- (#3,\z,#1);
      }
}

\newcommand{\xyzgrid}[6]
{ 
\foreach \x in {#1,...,#2} {
  \foreach \y in {#3,...,#4} {
      \draw [densely dotted] (\y, #5, \x) -- (\y, #6, \x);
  }
}
\foreach \y in {#3,...,#4} {
    \foreach \z in {#5,...,#6} {
      \draw [densely dotted] (\y, \z, #1) -- (\y, \z, #2);
    }
}
\foreach \x in {#1,...,#2} {
    \foreach \z in {#5,...,#6} {
      \draw [densely dotted] (#3, \z, \x) -- (#4, \z, \x);
    }
}
}

\theoremstyle{plain}
\newtheorem{thm}{Theorem}
\newtheorem{lem}[thm]{Lemma}
\newtheorem{prop}[thm]{Proposition}
\newtheorem{cor}[thm]{Corollary}

\theoremstyle{definition}
\newtheorem{definition}[thm]{Definition}
\newtheorem{exl}[thm]{Example}
\newtheorem{question}[thm]{Open Question}

\numberwithin{thm}{section}

\newcommand{\adj}{\leftrightarrow}
\newcommand{\adjeq}{\leftrightarroweq}

\def\Z{{\mathbb Z}}

\begin{document}
\title{Fundamental Groups and Euler Characteristics of Sphere-like Digital Images}
\author{Laurence Boxer
         \thanks{
    Department of Computer and Information Sciences,
    Niagara University,
    Niagara University, NY 14109, USA;
    and Department of Computer Science and Engineering,
    State University of New York at Buffalo.
    E-mail: boxer@niagara.edu
    }
    \and{P. Christopher Staecker
    \thanks{
    Department of Mathematics,
    Fairfield University,
    Fairfield, CT 06823-5195, USA.
    E-mail: cstaecker@fairfield.edu
    }
    }
}
\date{ }
\maketitle

\begin{abstract}
The current paper focuses on fundamental groups and Euler characteristics of various digital models of the 2-dimensional sphere. For all models that we consider, we show that the fundamental groups are trivial, and compute the Euler characteristics (which are not always equal). We consider the connected sum of digital surfaces and investigate how this operation relates to the fundamental group and Euler characteristic. We also consider two related but different notions of a digital image having ``no holes,'' and relate this to the triviality of the fundamental group.

Many of our results have origins in the paper~\cite{Han07} by S.-E. Han, which contains many errors. We correct these errors when possible, and leave some open questions. We also present some original results.


\end{abstract}

\section{Introduction}
A digital image is a graph that models an
object in a Euclidean space. In digital 
topology we study properties of digital
images analogous to the geometric and 
topological properties of the objects in 
Euclidean space that the images model. Among
these properties are digital versions of the
fundamental group and the Euler characteristic. 
The current paper focuses on fundamental groups and Euler characteristics of various digital models of the 2-dimensional sphere.

Most of our results were explored by Han in \cite{Han07}, where many errors appear. We correct almost all of these errors, leaving open some questions, and also obtain some new results. Many of the errors in \cite{Han07} result from inattention to basepoint preservation in homotopies of loops. The difference between pointed and unpointed homotopy turns out to be complex, and must be carefully considered. This issue has been explored in \cite{Staecker-etal} and \cite{BoSt}, and we continue that work in this paper. In particular, Example~\ref{contract-not-ptd} shows that contractibility does not imply pointed contractibility.
Errors also appear in the discussion of Euler characteristics in~\cite{Han07}. We correct these, many of which seem due to simple counting mistakes.


\section{Preliminaries}
\subsection{Fundamentals of digital topology}
\label{prelims}
Much of this section is quoted or paraphrased from other
papers in digital topology, such as~\cite{Boxer99,Boxer05,BoSt16}.

We will assume familiarity with the topological theory of digital images. See, e.g., \cite{Boxer94} for the standard definitions. All digital images $X$ are assumed to carry their own adjacency relations (which may differ from one image to another). When we wish to emphasize the particular adjacency relation we write the image as $(X,\kappa)$, where $\kappa$ represents
the adjacency relation.

Among the commonly used adjacencies are the $c_u$-adjacencies.
Let $x,y \in \Z^n$, $x \neq y$. Let $u$ be an integer,
$1 \leq u \leq n$. We say $x$ and $y$ are $c_u$-adjacent if
\begin{itemize}
\item There are at most $u$ indices $i$ for which 
      $|x_i - y_i| = 1$.
\item For all indices $j$ such that $|x_j - y_j| \neq 1$ we
      have $x_j=y_j$.
\end{itemize}
We often label a $c_u$-adjacency by the number of points
adjacent to a given point in $\Z^n$ using this adjacency.
E.g.,
\begin{itemize}
\item In $\Z^1$, $c_1$-adjacency is 2-adjacency.
\item In $\Z^2$, $c_1$-adjacency is 4-adjacency and
      $c_2$-adjacency is 8-adjacency.
\item In $\Z^3$, $c_1$-adjacency is 6-adjacency,
      $c_2$-adjacency is 18-adjacency, and $c_3$-adjacency
      is 26-adjacency.
\end{itemize}

\begin{definition}
\label{pathPtSet}
A subset $Y$ of a digital image $(X,\kappa)$ is
{\em $\kappa$-connected}~\cite{Rosenfeld},
or {\em connected} when $\kappa$
is understood, if for every pair of points $a,b \in Y$ there
exists a sequence $P=\{y_i\}_{i=0}^m \subset Y$ such that
$a=y_0$, $b=y_m$, and $y_i$ and $y_{i+1}$ are 
$\kappa$-adjacent for $0 \leq i < m$. $P$ is
then called a {\em path} from $a$ to $b$ in 
$Y$.
\end{definition}

The following generalizes a definition of
~\cite{Rosenfeld}.

\begin{definition}\label{continuous}
{\rm ~\cite{Boxer99}}
Let $(X,\kappa)$ and $(Y,\lambda)$ be digital images. A function
$f: X \rightarrow Y$ is $(\kappa,\lambda)$-continuous if for
every $\kappa$-connected $A \subset X$ we have that
$f(A)$ is a $\lambda$-connected subset of $Y$. 
\end{definition}

When the adjacency relations are understood, we will simply say that $f$ is \emph{continuous}. Continuity can be reformulated in terms of adjacency of points:
\begin{thm}
{\rm ~\cite{Rosenfeld,Boxer99}}
A function $f:X\to Y$ is continuous if and only if, for any adjacent points $x,x'\in X$, the points $f(x)$ and $f(x')$ are equal or adjacent. \qed
\end{thm}

See also~\cite{Chen94,Chen04}, where similar notions are referred to as {\em immersions}, {\em gradually varied operators},
and {\em gradually varied mappings}.


It is perhaps unfortunate that ``path'' is also used
with a meaning that is related to but distinct from the above. We will also use the following.

\begin{definition} (See {\rm \cite{Khalimsky}}.)
\label{dig-loop}
A {\em $\kappa-$path} in a digital image $X$ is
a $(2,\kappa)-$continuous function 
$f: [0,m]_{{\Z}} \rightarrow X$. If, further, $f(0) = f(m)$, we call
$f$ a {\em digital $\kappa-$loop}, and
the point $f(0)$ is the {\em basepoint} of the loop $f$.
If $f$ is a constant function, it is called a {\em trivial loop}. $\Box$
\end{definition}

Other terminology we use includes the following.
Given a digital image $(X,\kappa) \subset \Z^n$ and $x \in X$, the set of points adjacent to $x \in \Z^n$, the
neighborhood of $x$ in $\Z^n$, and the boundary of $X$
in $\Z^n$ are, respectively,
\[N_{\kappa}(x) = \{y \in \Z^n \, | \, y \mbox{ is }
    \kappa\mbox{-adjacent to }x\},\]
\[N_{\kappa}^*(x) = N_{\kappa}(x) \cup \{x\},
\]
and
\[ \delta_{\kappa}(X)=\{y \in X \, | \, 
    N_{\kappa}(y) \setminus X \neq \emptyset \}.
\]

\subsection{Digital homotopy}
Material appearing in this section is largely quoted
or paraphrased from other papers in digital topology.
See, e.g.,~\cite{BoSt1}.

A homotopy between continuous functions may be thought of as
a continuous deformation of one of the functions into the 
other over a finite time period.

\begin{definition}{\rm (\cite{Boxer99}; see also \cite{Khalimsky})}
\label{htpy-2nd-def}
Let $X$ and $Y$ be digital images.
Let $f,g: X \rightarrow Y$ be $(\kappa,\kappa')$-continuous functions.
Suppose there is a positive integer $m$ and a function
$F: X \times [0,m]_{{\Z}} \rightarrow Y$
such that

\begin{itemize}
\item for all $x \in X$, $F(x,0) = f(x)$ and $F(x,m) = g(x)$;
\item for all $x \in X$, the induced function
      $F_x: [0,m]_{{\Z}} \rightarrow Y$ defined by
          \[ F_x(t) ~=~ F(x,t) \mbox{ for all } t \in [0,m]_{{\Z}} \]
          is $(2,\kappa')-$continuous. That is, $F_x(t)$ is a path in $Y$.
\item for all $t \in [0,m]_{{\Z}}$, the induced function
         $F_t: X \rightarrow Y$ defined by
          \[ F_t(x) ~=~ F(x,t) \mbox{ for all } x \in  X \]
          is $(\kappa,\kappa')-$continuous.
\end{itemize}
Then $F$ is a {\rm digital $(\kappa,\kappa')-$homotopy between} $f$ and
$g$, and $f$ and $g$ are {\rm digitally $(\kappa,\kappa')-$homotopic in} $Y$.
If for some $x \in X$ we have $F(x,t)=F(x,0)$ for all
$t \in [0,m]_{{\Z}}$, we say $F$ {\rm holds $x$ fixed}, and $F$ is a {\rm pointed homotopy}.
$\Box$
\end{definition}

We denote a pair of homotopic functions as
described above by $f \simeq_{\kappa,\kappa'} g$.
When the adjacency relations $\kappa$ and $\kappa'$ are understood in context,
we say $f$ and $g$ are {\em digitally homotopic}
to abbreviate ``digitally 
$(\kappa,\kappa')-$homotopic in $Y$," and write
$f \simeq g$.

\begin{prop}
\label{htpy-equiv-rel}
{\rm ~\cite{Khalimsky,Boxer99}}
Digital homotopy is an equivalence relation among
digitally continuous functions $f: X \rightarrow Y$.
$\Box$
\end{prop}



\begin{definition}
{\rm ~\cite{Boxer05}}
\label{htpy-type}
Let $f: X \rightarrow Y$ be a $(\kappa,\kappa')$-continuous function and let
$g: Y \rightarrow X$ be a $(\kappa',\kappa)$-continuous function such that
\[ f \circ g \simeq_{\kappa',\kappa'} 1_X \mbox{ and }
   g \circ f \simeq_{\kappa,\kappa} 1_Y. \]
Then we say $X$ and $Y$ have the {\rm same $(\kappa,\kappa')$-homotopy type}
and that $X$ and $Y$ are $(\kappa,\kappa')$-{\rm homotopy equivalent}, denoted 
$X \simeq_{\kappa,\kappa'} Y$ or as
$X \simeq Y$ when $\kappa$ and $\kappa'$ are
understood.
If for some $x_0 \in X$ and $y_0 \in Y$ we have
$f(x_0)=y_0$, $g(y_0)=x_0$,
and there exists a homotopy between $f \circ g$
and $1_X$ that holds $x_0$ fixed, and 
a homotopy between $g \circ f$
and $1_Y$ that holds $y_0$ fixed, we say
$(X,x_0,\kappa)$ and $(Y,y_0,\kappa')$ are
{\rm pointed homotopy equivalent} and that $(X,x_0)$ 
and $(Y,y_0)$ have the 
{\rm same pointed homotopy type}, denoted 
$(X,x_0) \simeq_{\kappa,\kappa'} (Y,y_0)$ or as
$(X,x_0) \simeq (Y,y_0)$ when 
$\kappa$ and $\kappa'$ are understood.
$\Box$
\end{definition}

It is easily seen, from 
Proposition~\ref{htpy-equiv-rel}, that having the
same homotopy type (respectively, the same
pointed homotopy type) is an equivalence relation
among digital images (respectively, among pointed
digital images).

For $p \in Y$, we denote by $\overline{p}$ the constant
function $\overline{p}:X \rightarrow Y$ defined by $\overline{p}(x) = p$ for
all $x \in X$.

\begin{definition}
\label{htpy-trivial}
A digital image $(X,\kappa)$ is
$\kappa$-{\rm contractible \cite{Khalimsky,Boxer94}}
if its identity map is $(\kappa, \kappa)$-homotopic to a 
constant function $\overline{p}$ for some $p \in X$.
If the homotopy of the contraction holds $p$ fixed,
we say $(X, p, \kappa)$ is {\rm pointed $\kappa$-contractible}.
$\Box$
\end{definition}

When $\kappa$ is understood, we speak of {\em contractibility}
for short.
It is easily seen that $X$
is contractible
if and only if $X$
has the homotopy type 
of a one-point digital image.

The following is the first
example in the literature of a digital image that is
contractible but is not pointed contractible.

\begin{figure}
\begin{center}
\begin{tikzpicture} 
\xyzgrid{0}{2}{0}{2}{0}{0};

\foreach \x in {0,...,2} {
  \foreach \y in {0,2} {
    \foreach \z in {0,1} {
            \node at (\y,\z,\x) [vertex, fill=black] {};
}}}
\foreach \x in {0,...,2} {
  \foreach \y in {0,2} {
    \draw [densely dotted] (\y,0,\x) -- (\y,1,\x);
  }
}
\draw [densely dotted] (1,0,0) -- (1,1,0);
\draw [densely dotted] (1,0,2) -- (1,1,2);
\foreach \y in {0,2} {
  \draw [densely dotted] (\y,1,0) -- (\y,1,2);
}
\foreach \x in {0,2} {
  \draw [densely dotted] (0,1,\x) -- (2,1,\x);
  \foreach \y in {1} {
    \foreach \z in {0,1} {
      \node at (\y,\z,\x) [vertex, fill=black] {};
}}}
\node at (1,0,1) [vertex, fill=black] {};
\def\l{$x_0$}
\node at (0,1,0) [vertex, fill=black, label=above right:{$x_0$}] {};
\draw[->](0,0,0) to node[pos=.9, above] {$y$} (3,0,0);
\draw[->](0,0,0) to node[pos=.9, left] {$z$} (0,2,0);
\draw[->](0,0,0) to node[pos=.9, left] {$x$} (0,0,3);
\end{tikzpicture}
\end{center}
\caption{The image $X$ discussed in 
Example~\ref{contract-not-ptd}. All coordinates in this paper are ordered according to the axes in this Figure.
}
\label{contract-not-ptd-fig}
\end{figure}
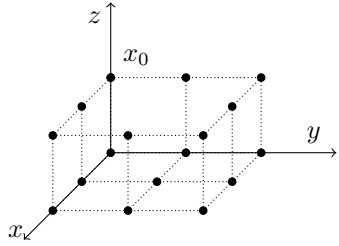

\begin{exl}
\label{contract-not-ptd} Let
$X=([0,2]_{\Z}^2 \times [0,1]_\Z) \setminus \{(1,1,1)\}$ (see  
Figure~\ref{contract-not-ptd-fig}). 
Let $x_0 =(0,0,1) \in X$. Then
$X$ is $6$-contractible, 
but $(X,x_0)$ is not pointed $6$-contractible.
\end{exl}

\begin{proof}
We show $X$ is $6$-contractible as follows. Let
$H: X \times [0,5]_{\Z} \to X$ be defined by
\[ H(x,0)=x;~~~ H(a,b,c,1)=(a,b,0); \]
\[H(a,b,c,t)=(a,\max\{0,b+1-t\},0) \mbox{ for } t \in \{2,3\}; \]
\[H(a,b,c,t)=(\max\{0,a+3-t\},0,0) \mbox{ for } t \in \{4,5\}. \]
It is easy to see that $H$ is a 6-contraction of $X$. 

Let \[P=([0,2]_{\Z}^2 \setminus \{(1,1)\}) \times \{1\}
\subset X.\]
Let $K: X \times [0,m]_{\Z} \to X$ be a 
$6$-contraction of $X$. As a simple closed curve of
more than 4 points, $P$ is not 
contractible~\cite{Boxer10},
so there exist $(a,b,1) \in P$ 
and some $t_0$ such that $K(a,b,1,t_0) \in [0,2]_\Z \times \{0\}$.
Since the induced function $K_{t_0}$ is $6$-continuous,
we must have $K_{t_0}(P) \subset [0,2]_\Z \times \{0\}$ 
(note this argument is 
suggested by the notion of ``path-pulling'' homotopy
discussed in~\cite{Staecker-etal}). In particular,
$K_{t_0}(x_0) \in [0,2]_{\Z} \times \{0\}$, so $K$
is not a pointed contraction. Since $K$ is
an arbitrary contraction of $X$, it follows that
$(X,x_0)$ is not pointed contractible.
\end{proof}

\begin{definition}
\label{nullhomotopy}
A continuous function $f: X \to Y$ is {\em nullhomotopic} in $Y$ if $f$ is homotopic in $Y$ to
a constant function. A pointed continuous function $f: (X,x_0) \to (Y,y_0)$ is {\em pointed nullhomotopic} in $Y$ if $f$ is pointed homotopic in $Y$ to
the constant function $\overline{y_0}$. A subset
$Y$ of $X$ is {\em nullhomotopic} in $X$ if the
inclusion map $i: Y \to X$ is nullhomotopic in $X$. A pointed subset
$(Y,x_0)$ of $(X,x_0)$ is {\em pointed nullhomotopic} in $X$ if the
inclusion map $i: (Y,x_0) \to (X,x_0)$ is pointed nullhomotopic in $X$.
\end{definition}

\subsection{Digital Loops}
Material in this section is largely quoted or
paraphrased from~\cite{Boxer06}.

If $f$ and $g$ are digital $\kappa-$paths in $X$
such that $g$ starts where $f$ ends,
the {\em product} (see~\cite{Khalimsky})
of $f$ and $g$, written $f \cdot g$, is, intuitively, the $\kappa-$path obtained
by following $f$ by $g$.  Formally, if $f: [0,m_1]_{{\Z}} \rightarrow X$,
$g: [0,m_2]_{{\Z}} \rightarrow X$, and $f(m_1)=g(0)$, then
$(f \cdot g): [0,m_1+m_2]_{{\Z}} \rightarrow X$ is defined by
\[(f \cdot g)(t) = \left\{
         \begin{array}{ll}
               f(t) & \mbox{if } t \in [0,m_1]_{{\Z}}; \\
               g(t - m_1) & \mbox{if } t \in [m_1,m_1+m_2]_{{\Z}}.
         \end{array}
        \right .  \]


It is undesirable to restrict homotopy classes of loops
to loops defined on the
same digital interval.  The following
notion of {\em trivial extension} allows a loop
to ``stretch" and remain in the same pointed homotopy class.
Intuitively, $f'$ is a trivial extension of $f$ if $f'$ follows the same
path as $f$, but more slowly, with pauses for rest
(subintervals of the domain on which $f'$ is constant).

\begin{definition}
\label{triv-extension}
{\rm \cite{Boxer99}}
Let $f$ and $f'$ be $\kappa-$paths in a digital image $X$. We say
$f'$ is a {\rm trivial extension of} $f$ if there are sets of $\kappa-$paths
$\{f_1, f_2, \ldots, f_k\}$ and $\{F_1, F_2, \ldots, F_p\}$ in $X$ such that
\begin{enumerate}
  \item $k \leq p$;
  \item $f = f_1 \cdot f_2 \cdot \ldots \cdot f_k$;
  \item $f' = F_1 \cdot F_2 \cdot \ldots \cdot F_p$; and
  \item there are indices
        $1 \leq i_1 < i_2 < \ldots < i_k \leq p$ such that
  \begin{itemize}
    \item $F_{i_j} = f_j$, $1 \leq j \leq k$, and
    \item $i \not \in \{i_1, i_2, \ldots, i_k\}$ implies $F_i$ is a
          trivial loop. $\Box$
  \end{itemize}
\end{enumerate}
\end{definition}

This notion allows us to compare the digital homotopy properties of loops
whose domains may have differing cardinality, since if $m_1 \leq m_2$,
we can obtain a trivial extension of a loop
$f:[0,m_1]_{{\Z}} \rightarrow X$ to
$f':[0,m_2]_{{\Z}} \rightarrow X$ via
\[ f'(t) =  \left\{
         \begin{array}{ll}
               f(t) & \mbox{if } 0 \leq t \leq m_1; \\
               f(m_1) &  \mbox{if } m_1 \leq t \leq m_2.
         \end{array}
        \right.  \]

We use the following notions to define the class of a pointed loop.

\begin{definition}
\label{holds-fixed}
Let $f, g: [0,m]_{{\Z}} \rightarrow (X, x_0)$ be digital loops with
basepoint $x_0$.
If $H: [0,m]_{{\Z}} \times [0,M]_{{\Z}} \rightarrow X$ is a
digital homotopy between $f$ and $g$
such that for all $t \in [0,M]_{{\Z}}$ we have
\[ H(0,t) ~=~ H(m,t), \]
we say $H$ {\rm is loop-preserving}. If, further,
for all $t \in [0,M]_{{\Z}}$ we have
\[H(0,t) ~=~ H(m,t) ~=~ x_0, \]
we say $H$ {\rm holds the endpoints fixed}. $\Box$
\end{definition}


Digital $\kappa-$loops $f$ and $g$ in X with the same basepoint $p$
{\em belong to the same $\kappa-$loop class in} $X$
if there are trivial extensions
$f'$ and $g'$ of $f$ and $g$, respectively, whose domains have the same
cardinality, and a homotopy between $f'$ and $g'$
that holds the endpoints fixed ~\cite{Boxer99}.

Membership in the same loop class in $(X,x_0)$ is
an equivalence relation among digital $\kappa-$loops
~\cite{Boxer99}.

We denote by $[f]$ the loop class of a loop $f$ in $X$.
We have the following.

\begin{prop}
\label{well-defined}
{\rm \cite{Boxer99,Khalimsky}}
Suppose $f_1,f_2,g_1,g_2$ are digital loops in a pointed digital
image $(X,x_0)$, with $f_2 \in [f_1]$ and
$g_2 \in [g_1]$.  Then $f_2 \cdot g_2 \in [f_1 \cdot g_1]$. $\Box$
\end{prop}

\subsection{Digital fundamental group}
Inspired by the fundamental group of a topological space,
several researchers~\cite{Stout,Kong,Boxer99,BoSt}
have developed versions of a fundamental
group for digital images. These are not all equivalent; 
however, it is shown in~\cite{BoSt} that the 
version of the fundamental group developed in that 
paper is equivalent to the version in~\cite{Boxer99}.
In this paper, we use the version of the fundamental
group developed in~\cite{Boxer99}.

Material appearing in this section is largely quoted
or paraphrased from other papers in digital topology.
See, e.g.,~\cite{Boxer99,Boxer06}.

Let $(X,p,\kappa)$ be a pointed digital image.
Consider the set
$\Pi_1^{\kappa}(X,p)$ of $\kappa$-loop classes $[f]$ in $X$ with
basepoint $p$.
By Proposition~\ref{well-defined}, the {\em product} operation
\[ [f] * [g]~=~[f \cdot g] \]
is well-defined on $\Pi_1^{\kappa}(X,p)$.
%
The operation $*$ is associative on $\Pi_1^{\kappa}(X,p)$~\cite{Khalimsky}.

\begin{lem}
\label{ident-elt}
{\rm \cite{Boxer99}}
Let $(X,p)$ be a pointed digital image.
Let ${\overline{p}}: [0,m]_{{\Z}} \rightarrow X$ be the constant function
$\overline{p}(t)=p$.  Then $[{\overline{p}}]$ is an identity element for
$\Pi_1^{\kappa}(X, p)$. $\Box$
\end{lem}

\begin{lem}
\label{inverse}
{\rm \cite{Boxer99}}
If $f: [0,m]_{{\Z}} \rightarrow X$ represents an element of $\Pi_1(X,p)$,
then the function $g: [0,m]_{{\Z}} \rightarrow X$ defined by
\[ g(t) = f(m-t) \mbox{ for } t \in [0,m]_{{\Z}} \]
is an element of $[f]^{-1}$ in $\Pi_1^{\kappa}(X,p)$. $\Box$
\end{lem}

\begin{thm}
{\rm \cite{Boxer99}}
$\Pi_1^{\kappa}(X,p)$ is a group under the $*$ product operation,
the {\em $\kappa$-fundamental group of} $(X,p)$. $\Box$
\end{thm}

It follows from the next result that in a connected digital image $X$,
the digital fundamental group is independent of the choice of basepoint.

\begin{thm}
\label{basept-indep}
{\rm \cite{Boxer99}}
Let $X$ be a digital image with adjacency relation $\kappa$. If
$p$ and $q$ belong to the same $\kappa-$component of $X$, then
$\Pi_1^{\kappa}(X,p)$ and $\Pi_1^{\kappa}(X,q)$ are isomorphic groups.
$\Box$
\end{thm}


Despite the existence of images that are
homotopy equivalent but not pointed homotopy
equivalent (\cite{Staecker-etal,BoSt}, Example~\ref{contract-not-ptd}),
notice that we do not require the homotopy equivalence in the following theorem to be a pointed homotopy
equivalence.

\begin{thm}
\rm{\cite{BoSt}}
\label{equiv-image-iso}
Let $f:(X,\kappa) \rightarrow (Y,\lambda)$ be a
$(\kappa,\lambda)$-homotopy equivalence of connected digital images.  Then $\Pi_1^{\kappa}(X,x_0)$ and
$\Pi_1^{\lambda}(Y,f(x_0))$ are isomorphic groups. $\Box$
\end{thm}

Similarly, in the following we do not require
pointed contractibility.

\begin{cor}
\label{contractible-trivial}
If $(X,\kappa)$ is a contractible digital image,
then $\Pi_1^{\kappa}(X,x_0)$ is a trivial group.
\end{cor}

\begin{proof} Since $X$ is contractible, $X$ is
homotopy equivalent to a one-point image,
which has a trivial fundamental group. The assertion
follows from Theorem~\ref{equiv-image-iso}.
\end{proof}

\section{Fundamental groups of 2-spheres}
\label{han-3-3}
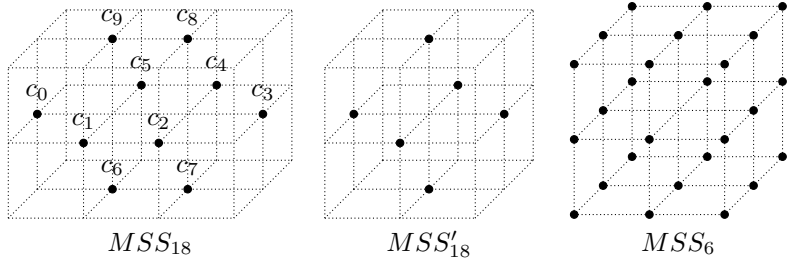
\begin{figure}

\begin{tabular}{ccc}
\begin{tikzpicture}
%
\def \pts{0/0/0/c_0,1/1/0/c_1,1/2/0/c_2,0/3/0/c_3,-1/2/0/c_4,-1/1/0/c_5,0/1/-1/c_6,0/2/-1/c_7,0/2/1/c_8,0/1/1/c_9};
\xyzgrid{-1}{1}{0}{3}{-1}{1}
\foreach \x/\y/\z/\l in \pts {
            \node at (\y,\z,\x) [vertex, fill=black, label=$\l$,] {};
}
\end{tikzpicture}%
&%
\begin{tikzpicture} 
%
\def \pts{0/0/0/, 1/1/0/, 0/2/0/, -1/1/0/, 0/1/-1/, 0/1/1/};
\xyzgrid{-1}{1}{0}{2}{-1}{1};
\foreach \x/\y/\z/\l in \pts {
            \node at (\y,\z,\x) [vertex, fill=black, label=$\l$,] {};
}
\end{tikzpicture}
&%
\begin{tikzpicture} 
%
\def \pts{0/0/0/,0/0/1/,0/0/2/,0/1/0/,0/1/1/,0/1/2/,0/2/0/,0/2/1/,0/2/2/,
              1/0/0/,1/0/1/,1/0/2/,1/1/0/,          1/1/2/,1/2/0/,1/2/1/,1/2/2/,
              2/0/0/,2/0/1/,2/0/2/,2/1/0/,2/1/1/,2/1/2/,2/2/0/,2/2/1/,2/2/2/}
\xyzgrid{0}{2}{0}{2}{0}{2};
\foreach \x/\y/\z/\l in \pts {
            \node at (\y,\z,\x) [vertex, fill=black, label=$\l$,] {};
}
\end{tikzpicture}
\\
$MSS_{18}$ & $MSS'_{18}$ & $MSS_6$
\end{tabular}
\caption{Three different digital images that model the 2-sphere}
\label{3surfs}
\end{figure}

The following digital images are considered in~\cite{Han07}.
Each in some sense models the 2-dimensional
sphere $S^2$ in Euclidean 3-space
(see Figure~\ref{3surfs}).
\begin{itemize}
\item $MSS_{18} = \{c_i\}_{i=0}^9$, where
\[c_0=(0,0,0), c_1=(1,1,0), c_2=(1,2,0),c_3=(0,3,0),
  c_4=(-1,2,0), \]
\[ c_5=(-1,1,0),
 c_6=(0,1,-1),c_7=(0,2,-1),c_8=(0,2,1),c_9=(0,1,1).
\]

\item $MSS'_{18}=MSS'_{26}=\\ \{(0,0,0), (1,1,0),(0,2,0), (-1,1,0), (0,1,-1), (0,1,1)\}$

\item $MSS_6=[0,2]_{\Z}^3 \setminus \{(1,1,1)\}$.
\end{itemize}




In this section we will show that fundamental groups $\Pi_1^k$ for $k\in\{6,18,26\}$ for each of these images are trivial groups. These computations are attempted in \cite[Lemma 3.3]{Han07}, but in each case the argument is incorrect or incomplete. 

We begin with $MSS_{18}$, which is simplest for $k=6$.
\begin{prop}
\label{Pi1MSS18}
Let $x \in MSS_{18}$. Then $\Pi_1^6(MSS_{18}, x)$ is trivial. 
\end{prop}
\begin{proof}
The 6-component of $x$ in $MSS_{18}$ is $\{x\}$.
Thus, every 6-loop in $(MSS_{18},x)$ is a trivial
loop, and the assertion follows.
\end{proof}

For $18$-adjacency we will use the following lemma:

\begin{lem}\label{avoidpoint}
Let $f:[0,m]_\Z \to MSS_{18}$ be a $c_0$-based 18-loop. Then there is a $c_0$-based 18-loop $f': [0,m]_\Z \to MSS_{18} \setminus \{c_3\}$ with $[f]=[f']$ in $\Pi_1^{18}(MSS_{18},c_0)$. 
\end{lem}
\begin{proof}
We may assume that $f$ is not a trivial extension of another loop. Let $t$ be the minimal number with $f(t)=c_3$. Since $f$ is a loop based at $c_0$, and $c_0$ is distant from $c_3$ in $MSS_{18}$, we must have $3 \le t \le m-3$. Since $f$ is not a trivial extension of another loop, we must have $f(t-1) \adj c_3 \adj f(t+1)$, where $\adj$ means 18-adjacent (and not equal). 

Examining the structure of $MSS_{18}$ we see that there will always be some point $c \in \{c_2,c_4,c_8,c_7\}$ with $f(t-1) \adjeq c \adjeq f(t+1)$, where $\adjeq$ means ``equal or 18-adjacent.'' Now define $f_1:[0,m]_\Z \to MSS_{18}$ by:
\[ 
f_1(x) = \begin{cases}
f(x) & \text{ if } x \neq t, \\
c & \text{ if } x = t.
\end{cases} \]
By our choice of $c$, this $f_1$ will be continuous, and is 18-homotopic to $f$ in one time step. Because $3\le t \le m-3$, this homotopy holds the endpoints fixed. Thus $[f]=[f_1]$ in $\Pi_1^{18}(MSS_{18})$.

Note that $f_1$ meets the point $c_3$ one time fewer than $f$ does. By applying the  construction above, again, to $f_1$, we obtain a loop $f_2$ with $[f_2]=[f]$ which meets the point $c_3$ two fewer times than $f$ does. Iterating this construction eventually gives a loop $f'$ with $[f']=[f]$ which never meets $c_3$, as desired.
\end{proof}

The following statement corresponds to \cite[Lemma 3.3 (1)]{Han07}. The argument given for proof in \cite{Han07} merely demonstrates a specific 18-loop in $MSS_{18}$ and shows that it is contractible, rather than showing that all such loops are contractible holding the endpoints fixed.

\begin{prop}
\label{correct-3-3-1}
Let $x\in MSS_{18}$. Then $\Pi_1^{18}(MSS_{18},x)$ is a trivial
group.
\end{prop}
\begin{proof}
It suffices to consider the case where $x=c_0$. Let $f:[0,M]_\Z \to MSS_{18}$ be a loop based at $c_0$.
We will show that $[f]=[\bar c_0]$, where $\bar c_0$ is the constant path at $c_0$. By Lemma \ref{avoidpoint} we have a loop $f'$ with $[f]=[f']$ and $c_3 \not\in f'([0,M]_\Z)$. 

For $t\in \Z$, let $Q_t: \Z^3 \to \Z^3$ be defined by $Q_t(a,b,c) = (a, \min(b,t), c)$. Since $f'$ avoids $c_3$, we have $Q_2\circ f' = f'$. It is also clear that $f' \simeq Q_1 \circ f'$ by a one-step homotopy, and that $Q_1\circ f'$ is a path meeting only points that are adjacent or equal to $c_0$. Thus $Q_1 \circ f' \simeq \bar c_0$ by a one-step homotopy, where $\bar c_0$ is the constant path at $c_0$. All of these homotopies fix the basepoint, and so we have $[f]=[f']=[\bar c_0]$ as desired.
\end{proof}

Our final case for $MSS_{18}$ uses $26$-adjacency. Informally, since we have already shown that all loops in $MSS_{18}$ are 18-contractible, it should follow that they are 26-contractible, since any 18-contraction is also automatically a 26-contraction.
Further, we can easily see that every 26-loop in $MSS_{18}$ is 26-homotopic in 1 step to an 18-loop in $MSS_{18}$
with the homotopy holding the endpoints fixed.
This intuition leads to the following lemma which is quite general. 
Below, $\adjeq_\kappa$ means ``$\kappa$-adjacent or equal''. 
\begin{lem}\label{pi1-different-adj}
Let $X$ be a digital image with two adjacency relations $\kappa$ and $\lambda$. Assume that if $x\adjeq_\kappa y$ then $x \adjeq_\lambda y$, and that if $x\adjeq_\lambda y$ then there is a $\kappa$-path in $X$ of length 2 from $x$ to $y$. If $\Pi_1^\kappa(X,x_0)$ is trivial for some $x_0\in X$, then $\Pi_1^\lambda(X,x_0)$ is trivial.
\end{lem}

\begin{proof}
Let $h: [0,m]_{\Z} \to X$ be a
$\lambda$-loop based at $x_0$.
Let $\bar h: [0,2m]_{\Z} \to X$ be
the trivial extension of $h$ defined
by
\[ \bar h(s) = \left \{ \begin{array}{ll}
    h(s/2) & \mbox{if } s \mbox{ is even;} \\
    \bar h(s-1) & \mbox{if } s \mbox{ is odd.}
\end{array} \right .
\]
Since $\bar h(2u) \adjeq_{\lambda} \bar h(2(u+1))$, 
there is a $\kappa$-path, which is therefore a $\lambda$-path,
of length 2
from $\bar h(2u)$ to $\bar h(2(u+1))$ through some $x_u \in X$. Therefore,
the function
$H: [0,2m]_{\Z}\times [0,1]_{\Z} \to X$ defined by
\[ H(s,t)= \left \{ \begin{array}{ll}
  \bar h(s) & \mbox{if } s \mbox{ is even;} \\
  \bar h(s) & \mbox{if } s \mbox{ is odd and } t=0; \\
  x_u & \mbox{if } s=2u+1 \mbox{ and } t=1,
\end{array} \right .
\]
is a $\lambda$-homotopy from $\bar h$ to a
$\kappa$-loop $h'$ that keeps the
endpoints fixed. Therefore, we have
\begin{equation}
\label{h-h'}
[h]_{\lambda} = [\bar h]_{\lambda}
= [h']_{\lambda},
\end{equation}
where the subscript $\lambda$ indicates that we are considering the loop class in $\Pi_1^\lambda(X,x_0)$. 

Since $\Pi_1^{\kappa}(X,x_0)$ is
trivial, there is a trivial extension
$h''$ of $h'$ that is $\kappa$-homotopic, hence $\lambda$-homotopic, to
a trivial loop $\overline{x_0}$ keeping the
endpoints fixed. Therefore,
$[h']_{\lambda}=[h'']_{\lambda}=[\overline{x_0}]_{\lambda}$. With equation~(\ref{h-h'}), this implies
$[h]_{\lambda}=[\overline{x_0}]_{\lambda}$.
The assertion follows.
\end{proof}

The hypothesis above concerning paths of length 2 could possibly be weakened, but some form of this restriction is necessary. As a counterexample to a more general statement, consider the example in Figure \ref{48example}. Here $X\subset \Z^2$, and 4-adjacency implies 8-adjacency, but $\Pi_1^4(X)$ is trivial and $\Pi_1^8(X)$ is infinite cyclic.
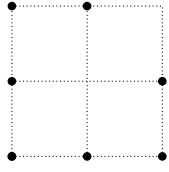
\begin{figure}
\[
\begin{tikzpicture}
\foreach \x in {0,...,2} {
  \draw[densely dotted] (\x,0) -- (\x,2);
  \draw[densely dotted] (0,\x) -- (2,\x);
  \node at (\x,0) [vertex, fill=black] {};
  \node at (0,\x) [vertex, fill=black] {};
  }
\node at (1,2) [vertex, fill=black] {};
\node at (2,1) [vertex, fill=black] {};
\end{tikzpicture}
\]
\caption{4-adjacency implies 8-adjacency, but $\Pi_1^4(X)$ is trivial while $\Pi_1^8(X)$ is not.\label{48example}}
\end{figure}

The lemma immediately leads to the following (which does not appear in \cite{Han07}): 

\begin{prop}
Let $x\in MSS_{18}$. Then $\Pi_1^{26}(MSS_{18},x)$ is a trivial group.
\end{prop}
\begin{proof}
We will apply Lemma \ref{pi1-different-adj} with $\kappa = 18$ and $\lambda=26$. Observe that any two 26-adjacent points of $MSS_{18}$ can be connected by a $18$-path of length 2. Then since $\Pi_1^{18}(MSS_{18})$ is trivial, Lemma \ref{pi1-different-adj} shows that $\Pi_1^{26}(MSS_{18})$ is trivial.
%
\end{proof}

Since $MSS_{18}'=MSS_{26}'$, Proposition~\ref{MSS18'summary} below encompasses \cite[Lemma 3.3~(2), (4)]{Han07}, and agrees with the assertions given in that paper. The proofs given in \cite{Han07} are incomplete. The argument for Lemma 3.3 (2) claims only that $MSS'_{18}$ is contractible, implying the use of a result like Corollary \ref{contractible-trivial} to complete the proof. No result like Corollary \ref{contractible-trivial} appears in \cite{Han07}; in fact our proof of Corollary \ref{contractible-trivial} depends on a nontrivial recent result in~\cite{BoSt}.
The argument for Lemma 3.3 (4) in \cite{Han07} merely asserts without proof that every 6-loop in $MSS_k'$ is 6-nullhomotopic; further, the argument neglects to require such nullhomotopies to fix the endpoints.
\begin{prop}
\label{MSS18'summary}
Let $x\in MSS'_{18}$. Then $\Pi_1^k(MSS_{18}',x)$ is a trivial group, $k \in \{6,18,26\}$.
\end{prop}

\begin{proof}

For $k=6$, the 6-component of $x$ in $MSS_{18}'$ is $\{x\}$.
Thus, every 6-loop in $(MSS_{18}',x)$ is a trivial
loop, and the assertion follows.

For $k\in \{18,26\}$, we observe that 
Proposition~4.1 of~\cite{Boxer06}  
shows that $MSS_{18}'$ is 26-contractible; indeed, its
proof shows that $MSS_{18}'$ is pointed 26-contractible;
and the same argument shows that $MSS_{18}'$ is pointed 
18-contractible. The assertion follows from 
Corollary~\ref{contractible-trivial}.
\end{proof}

\cite{Han07} states incorrectly in Lemma~3.3~(3) that $\Pi^6_1(MSS_6)$ is a free group with two generators. Specifically it is claimed that the following equatorial loop represents a nontrivial generator of $\Pi_1^6(MSS_6)$:
\[ D = ((0,0,1),(1,0,1),(2,0,1),(2,1,1),(2,2,1),(1,2,1),(0,2,1),(0,1,1),(0,0,1)) \]
But $D$ is in fact trivial in $\Pi_1^6(MSS_6)$. Consider the following sequence of loops, starting with a trivial extension of $D$:
{\footnotesize
\begin{align*}
((0,0,1),(0,0,1),(1,0,1),(2,0,1),(2,1,1),(2,2,1),(1,2,1),(0,2,1),(0,1,1),(0,0,1),(0,0,1)) \\
((0,0,1),(0,0,2),(1,0,2),(2,0,2),(2,1,2),(2,2,2),(1,2,2),(0,2,2),(0,1,2),(0,0,2),(0,0,1)) \\
((0,0,1),(0,0,2),(1,0,2),(2,0,2),(2,1,2),(2,1,2),(1,1,2),(0,1,2),(0,1,2),(0,0,2),(0,0,1)) \\
((0,0,1),(0,0,2),(1,0,2),(2,0,2),(2,0,2),(2,0,2),(1,0,2),(0,0,2),(0,0,2),(0,0,2),(0,0,1)) \\
((0,0,1),(0,0,2),(1,0,2),(1,0,2),(1,0,2),(1,0,2),(1,0,2),(0,0,2),(0,0,2),(0,0,2),(0,0,1)) \\
((0,0,1),(0,0,2),(0,0,2),(0,0,2),(0,0,2),(0,0,2),(0,0,2),(0,0,2),(0,0,2),(0,0,2),(0,0,1)) \\
((0,0,1),(0,0,1),(0,0,1),(0,0,1),(0,0,1),(0,0,1),(0,0,1),(0,0,1),(0,0,1),(0,0,1),(0,0,1)) \\
\end{align*}}
This sequence of loops gives a homotopy that holds the endpoints fixed, from a trivial extension of $D$ to a trivial loop, and thus $D$ represents a trivial fundamental group element. In fact, arguments similar to those used in Lemma \ref{avoidpoint} and Proposition \ref{correct-3-3-1} can be repeated  for $MSS_6$: any 6-loop can be first moved to avoid the point $(1,2,1)$, and then composed with $Q_t$ to contract fully. We obtain the following, which also follows
as a case of Theorem~3.1 of the paper~\cite{Boxer06}:
\begin{prop}
Let $x \in MSS_6$. Then $\Pi^6_1(MSS_6,x)$ is a trivial group.
\end{prop}

Immediately we can also compute the other fundamental groups of $MSS_6$ (this result does not appear in \cite{Han07}):
\begin{prop}
Let $x\in MSS_6$. Then $\Pi_1^k(MSS_6,x)$ is a trivial group for $k\in \{18,26\}$. 
\end{prop}
\begin{proof}
For $k=18$, apply Lemma \ref{pi1-different-adj} using $\kappa=6$ and $\lambda=18$.

For $k=26$ we apply Lemma \ref{pi1-different-adj} using $\kappa=18$ and $\lambda = 26$.
\end{proof}

\section{Fundamental groups for connected sums of
digital 2-spheres}
\label{HanThm3-4}

Han~\cite{Han06} defines the {\em connected sum} of two digital surfaces $X$ and
$Y$, denoted $X \sharp Y$. Roughly, the
idea behind this operation is that one
removes the interior of a minimal (depending on the adjacency used) simple closed curve from each of $X$
and $Y$ such that there is an isomorphism $F$ between
these simple closed curves, and sews together the 
remainders of $X$ and $Y$ along these simple close
curves by identifying points that are matched by $F$. This minimal simple closed curve, together with its interior, is denoted $A_k$, and called a \emph{digital disk}. \cite{Han07} uses three different digital disks, shown in Figure \ref{hanDisks}. See~\cite{Han06} for details of
the definition of the $\sharp$ operation.

Combining the spheres from Figure \ref{3surfs} by the $\sharp$ operation gives new digital images, shown in Figure \ref{MSS6sharpMSS6}. In this section we show that both of these images have trivial fundamental groups.

Theorem~3.4(1) of~\cite{Han07} asserts
that $\Pi_1^6(MSS_6 \sharp MSS_6)$ is
a group with two generators, but this is a propagation of errors from the computation of $\Pi_1^6(MSS_6)$ in that paper.
Theorem~3.4(2) of~\cite{Han07} says $\Pi_1^{18}(MSS_{18}\sharp MSS_{18})$ is trivial but justifies it only by showing that a single 18-loop is contractible.
In fact, following again the arguments used for Lemma \ref{avoidpoint} and Proposition \ref{correct-3-3-1} we obtain a
correct proof for the following,
corresponding to Theorem~3.4(1) and
Theorem~3.4(2) of~\cite{Han07}.

\begin{thm}
\label{MSS6sharpSq}
$\Pi_1^6(MSS_6 \sharp MSS_6)$ and $\Pi_1^{18}(MSS_{18}\sharp MSS_{18})$ are trivial groups. $\Box$
\end{thm}

\begin{figure}
\begin{center}
\begin{tabular}{ccc}
\begin{tikzpicture}
\foreach \x in {-1,...,1} {
  \foreach \y in {0,1} {
     \node at (\x,\y) [vertex, fill=black] {};
}}
\node at (0,-1) [vertex, fill=black] {};
\node at (0,2) [vertex, fill=black] {};
\draw (0,-1) -- (1,0) -- (1,1) -- (0,2) -- (-1,1) -- (-1,0) -- cycle;
\draw (0,-1) -- (0,2);
\foreach \y in {0,1} {
	\draw (-1,\y) -- (1,\y);
	}
\draw (-1,1) -- (0,0) -- (1,1);
\draw (-1,0) -- (0,1) -- (1,0);
\end{tikzpicture}
&
\begin{tikzpicture}
\def \pts{1/0, 0/1, -1/0, 0/-1}
\foreach \x/\y in \pts {
            \node at (\x,\y) [vertex, fill=black] {};
}
\node at (0,0) [vertex, fill=black] {};
\draw (0,1) -- (0,-1);
\draw (1,0) -- (-1,0);
\draw (1,0) \foreach \x\y in \pts { -- (\x,\y)} -- cycle;
\end{tikzpicture}
&
\begin{tikzpicture}
\foreach \x in {0,...,2} {
  \draw (\x,0) -- (\x,2);
  \foreach \y in {0,...,2} {
            \node at (\x,\y) [vertex, fill=black] {};
}}
\foreach \y in {0,...,2} {
  \draw (0,\y) -- (2,\y);
}
\end{tikzpicture}
\\
$MSC^*_8$ & $MSC'^*_8$ & $MSC^*_4$
\end{tabular}
\end{center}
\caption{Various ``digital disks'' that are used to form connected sums}
\label{hanDisks}
\end{figure}
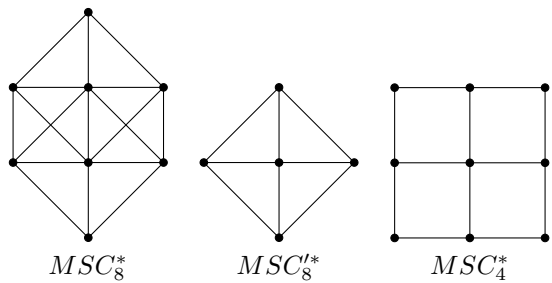

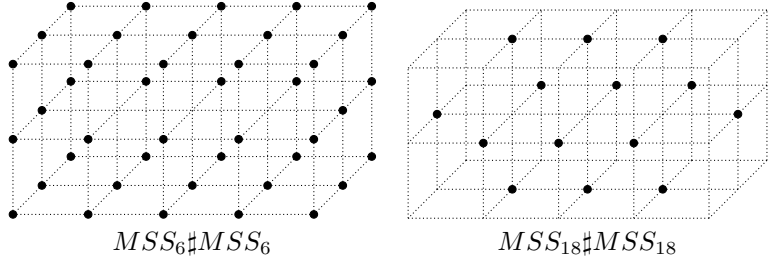
\begin{figure}
\begin{center}\begin{tabular}{cc}
\begin{tikzpicture} 
\xyzgrid{0}{2}{0}{4}{0}{2};
\foreach \x in {0,...,2} {
  \foreach \y in {0,...,4} {
    \foreach \z in {0,2} {
            \node at (\y,\z,\x) [vertex, fill=black] {};
}}}
\foreach \x in {0,2} {
  \foreach \y in {0,...,4} {
    \foreach \z in {1} {
            \node at (\y,\z,\x) [vertex, fill=black] {};
}}}
\foreach \x/\y/\z in {1/0/1, 1/4/1} {
            \node at (\y,\z,\x) [vertex, fill=black] {};
}
%
\end{tikzpicture}
&
\begin{tikzpicture} 
\xyzgrid{0}{2}{0}{4}{0}{2};
\foreach \x in {0,2} {
  \foreach \y in {1,...,3} {
    \foreach \z in {1} {
            \node at (\y,\z,\x) [vertex, fill=black] {};
}}}
\foreach \x in {1} {
  \foreach \y in {1,...,3} {
    \foreach \z in {0,2} {
            \node at (\y,\z,\x) [vertex, fill=black] {};
}}}
\foreach \x/\y/\z in {1/0/1, 1/4/1} {
            \node at (\y,\z,\x) [vertex, fill=black] {};
}
\end{tikzpicture} \\
$MSS_6\sharp MSS_6$ & $MSS_{18} \sharp MSS_{18}$
\end{tabular}
\end{center}
\caption{Connected sums of some images from Figure \ref{3surfs}.}
\label{MSS6sharpMSS6}
\end{figure}



Theorem~3.4(3) and Theorem~3.4(4) of~\cite{Han07} both make assertions
about $\Pi_1^{18}(SS_{18})$.
However, $SS_{18}$ is not defined in~\cite{Han07}.

If ``$SS_{18}$'' is interpreted as ``$MSS_{18}$'', then Theorem~3.4(3) 
of~\cite{Han07} would say that
$\Pi_1^{18}(MSS_{18} \sharp MSS_{18})$ is
isomorphic to $\Pi_1^{18}(MSS_{18})$. This
assertion would be redundant in light of
Lemma~3.3(1) of~\cite{Han07}, which (as corrected above
at Proposition~\ref{Pi1MSS18}) says
$\Pi_1^{18}(MSS_{18})$ is trivial, and Theorem~3.4(2) of~\cite{Han07} (as corrected above at Theorem~\ref{MSS6sharpSq}), which
says that $\Pi_1^{18}(MSS_{18} \sharp MSS_{18})$
is a trivial group.

Again interpreting ``$SS_{18}$'' as ``$MSS_{18}$'', then Theorem~3.4(4) 
of~\cite{Han07} would say that
$\Pi_1^{18}(MSS_{18}' \sharp MSS_{18})$ is isomorphic to
$\Pi_1^{18}(MSS_{18})$. This  follows since each is trivial. That $\Pi_1^{18}(MSS_{18})$ is trivial is shown
above in Proposition~\ref{correct-3-3-1}. That
$\Pi_1^{18}(MSS_{18}' \sharp MSS_{18})$ is trivial
follows from an argument similar to that used to prove
Theorem~\ref{MSS6sharpSq}, or by
observing that $MSS_{18}' \sharp MSS_{18}$ is (18,18)-isomorphic to
$MSS_{18}$.

Theorem~3.4(5) of~\cite{Han07} says that
$\Pi_1^k(MSS_k' \sharp MSS_k')$ is a
trivial group for $k \in \{18,26\}$. The
assertion is correct, and the argument given
in proof is basically correct; its only
flaw is in its dependence on the incorrectly
proven Lemma~3.3(2) of~\cite{Han07}. Since 
our Proposition~\ref{MSS18'summary} 
gives a correct proof of Lemma~3.3(2)
of~\cite{Han07}, we can accept the assertion
of Theorem~3.4(5) of~\cite{Han07}.

Theorem~3.4(6) of~\cite{Han07} makes an assertion
about $\Pi_1^{18}(SS_{26})$.
However, $SS_{26}$ isn't defined in~\cite{Han07}. If we interpret ``$SS_{26}$'' as ``$MSS_{26}'$'' then Han's Theorem~3.4(6) would say that 
$\Pi_1^{26}(MSS_{26}' \sharp MSS_{26}')$ is
isomorphic to $\Pi_1^{26}(MSS_{26}')$. This assertion can be correctly proven by observing
that 
$(MSS_{26}' \sharp MSS_{26})$
is $(26,26)$-isomorphic to
$MSS_{26}'$.

\section{Fundamental groups for images without holes}
\label{holes-sec}
In~\cite{Han07}, attempts are made to derive fundamental 
groups for certain digital surfaces without holes.
Errors in these efforts are discussed in this section. We also
obtain some related original results.


\begin{definition}
\rm{\cite{Han07}}
\label{hanHole}
A digital image $(X,\kappa)$ {\em has no 
$\kappa$-hole} if every $\kappa$-path in $X$ 
is $\kappa$-nullhomotopic in $X$.
\end{definition}

In the definition above, we must understand ``path'' in the sense of Definition~\ref{pathPtSet}, as any path in the sense of Definition~\ref{dig-loop} is nullhomotopic. Recall from Definition~\ref{nullhomotopy} that a path in the sense of Definition \ref{pathPtSet} is nullhomotopic when its inclusion map is nullhomotopic. We show below that, in the case where each component of $X$ is finite, the no hole condition is equivalent to contractibility of each component.

Using Definition~\ref{hanHole}, it is claimed as Theorem~3.5 
of~\cite{Han07} that a closed
$k$-surface $X \subset \Z^3$ with no
$k$-holes has trivial fundamental group for
$k \in \{18,26\}$. However, the argument given
fails to require homotopies between
loops to hold the endpoints fixed.

By Definition~\ref{htpy-2nd-def}, a condition that
is necessary for a connected image $X$ to be
contractible or to have
no holes is that $X$ must have a finite upper bound for lengths of 
shortest paths between distinct points, since there
are finitely many ``time steps'' in a homotopy. We will use the following.

\begin{prop}
\label{connected-path}
Let $(X,\kappa)$ be a digital image. Then $X$ is finite and connected if and only if $X$ is a $\kappa$-path.
\end{prop}

\begin{proof} 
First assume that $X$ is finite and connected. 
Let $X=\{x_i\}_{i=0}^m$. Since
$X$ is connected, there is a path
$P_i$ in $X$ from $x_{i-1}$ to $x_i$,
$i \in \{1,2,\ldots,m\}$. By traversing
$P_1$ followed by $P_2$ followed by $\ldots$
followed by $P_m$, we see that $X$ is the
path $\bigcup_{i=1}^m P_i$.

The converse is clear from the definition of path and connectivity - any path must be finite and connected.
\end{proof}

\begin{prop}
\label{noHoles-Contract}
Let $(X,\kappa)$ be a digital image such that each component is finite.
Then $X$ has no $\kappa$-hole if and
      only if every component of $X$ is 
      $\kappa$-contractible.
\end{prop}

\begin{proof}
Suppose $X$ has no $\kappa$-hole. Let $A$ be
a $\kappa$-component of $X$. Then $A$ is
finite, and by
Proposition~\ref{connected-path}, $A$ is a 
$\kappa$-path. Since $X$ has no $\kappa$-hole, the inclusion $i:A\to X$ is nullhomotopic in~$A$, and thus $A$ is contractible.

Conversely, suppose every component of
$X$ is $\kappa$-contractible.  Since every
path $P \subset X$ is a connected set, we
must have $P$ contained in some component
$A$ of $X$. By restricting a contraction of
$A$ to $P$, we have a nullhomotopy of $P$ in
$X$. Thus, $X$ has no $\kappa$-hole.
\end{proof}

The importance of the finiteness restriction in Proposition~\ref{noHoles-Contract} is demonstrated in the following example.

\begin{exl}
\label{infinite}
$\Z$ has no 2-hole, but is connected and not 2-contractible.
\end{exl}

\begin{proof} Let $P=\{y_i\}_{i=0}^m$ be a 2-path in $\Z$, in the sense of Definition~\ref{pathPtSet}. Then $P$ is a digital interval: $P=[a,b]_{\Z}$. Therefore, the function
$H: P \times [0,b-a]_{\Z} \to \Z$ defined by
\[ H(s,t) = p(\max\{0,s-t\})
\]
is a nullhomotopy of $P$. It follows that $\Z$ has no 2-holes.

Clearly, $\Z$ is 2-connected.
$\Z$ is not 2-contractible, as given
$x \in \Z$, there is no finite bound
on the length of 2-paths from $y \in \Z$ to $x$, and a homotopy has only
finitely many steps in which the distance between points can be lessened by at most 2.
\end{proof}

The following is our modified and corrected version of Theorem~3.5 of~\cite{Han07}. The result in that paper is stated only for closed digital surfaces (which are automatically finite), but our theorem holds more generally for any digital image with finite components.

\begin{thm}\label{noholepi1}
If $(X,\kappa)$ has no $\kappa$-hole and each component of $X$ is finite, then $\Pi_1^\kappa(X,x_0)$ is trivial for any $x_0$. 
\end{thm}
\begin{proof}
Let $x_0 \in X$. By Proposition~\ref{noHoles-Contract},
the component $A$ of $x_0$ in $X$ is contractible. It
follows from Corollary~\ref{contractible-trivial} that
$\Pi_1^{\kappa}(X,x_0)=\Pi_1^{\kappa}(A,x_0)$ is trivial.
\end{proof}

%

A closely related version of the ``no hole'' condition can be formulated in terms of paths viewed as functions according to Definition \ref{dig-loop}.

\begin{definition}
\label{noloophole}
A digital image $(X,\kappa)$ {\em has no 
loophole} if every $\kappa$-loop in $X$ 
is $\kappa$-nullhomotopic in $X$ by a loop-preserving homotopy.
\end{definition}

As with Han's ``no hole'' condition, we can show that a space with no loopholes has trivial fundamental group. The following is another version of Theorem~3.5 of~\cite{Han07}.

\begin{thm}\label{noloopholepi1}
If $(X,\kappa)$ has no loopholes, then $\Pi_1^\kappa(X,x_0)$ is trivial for any $x_0$. 
\end{thm}
\begin{proof}
Let $f:[0,k]_{\Z} \to X$ be a loop in $X$ based at $x_0$. We must show that $[f]=[\bar x_0]$, where $\bar x_0$ is the constant loop at $x_0$. Since $X$ has no loopholes, $f$ is homotopic to $\bar x_0$ by a loop-preserving homotopy, say $H:[0,k]_{\Z}\times [0,m]_{\Z} \to X$. Since $H$ is loop-preserving, we have $H(0,t) = H(k,t)$ for each $t$. Let $p(t) = H(0,t)$, so $p$ is the path taken by the basepoints of the loops during the homotopy. Since both $f$ and $\bar x_0$ have basepoint $x_0$, this path $p$ is a loop at $x_0$. 

For $t \in [0,m]_{\Z}$, let $p_t: [0,m]_{\Z} \to X$ be defined by
$p_t(s) = p(\min\{s,t\})$. Then $p_t$ is a
path from $p(0)=x_0$ to $p(t)$ and
let $p_t^{-1}$ be the reverse path.

Let $\bar x_0$ be a constant path of length $m$. 
Let $\bar H: [0,k+2m]_{\Z} \times [0,m]_{\Z} \to X$ be defined by
\[ \bar H(s,t)=
(p_t \cdot H_t \cdot p_t^{-1})(s), \]
where $H_t(s)=H(s,t)$.
Then $\bar H$ is a homotopy from $p_0 \cdot f \cdot p_0^{-1}$ 
$= \bar x_0 \cdot f \cdot \bar x_0$,
a trivial extension of $f$,
to $p \cdot \bar x_0 \cdot p^{-1}$,
a trivial extension of $p \cdot p^{-1}$, and $\bar H$ holds the endpoints fixed. Therefore,
\begin{equation}
\label{f-forwardandback}
[f] = [p \cdot p^{-1}].
\end{equation}

Since the function
$K: [0,2m]_{\Z} \times [0,m]_{\Z} \to X$ defined by
\[ K(s,t)= (p_t \cdot p_t^{-1})(s)
\]
is a homotopy from $p \cdot p^{-1}$ to $\bar x_0 \cdot \bar x_0$ that holds the endpoints fixed, we have
\begin{equation}
\label{forwardandback-x0}
[p \cdot p^{-1}] = [\bar x_0].
\end{equation}
From equations~(\ref{f-forwardandback}) and~(\ref{forwardandback-x0}), we have $[f]=[\bar x_0]$, as desired.
\end{proof}

In the proof of Theorem~\ref{noloopholepi1} we also proved the following.
\begin{lem}
\label{tricky}
Let $f$ be a loop in $X$ based at $x_0$. If $f$ is homotopic to the constant loop $\bar x_0$ by a loop preserving homotopy, then $[f]=[\bar x_0]$ in $\Pi_1^\kappa(X,x_0)$. $\Box$
\end{lem}
Note that Lemma~\ref{tricky} need not be true when the constant loop $\bar x_0$ is replaced by some other loop. See the discussion following Definition 2.8 of \cite{Boxer06b} for an example of two loops that are homotopic by a loop-preserving homotopy but are not equivalent in the fundamental group.

\begin{figure}
\[
\begin{tikzpicture} 
%

\draw[] (0,0,4) -- (0,0,0) -- (4,0,0);
\draw[] (0,0,0) -- (0,4,0);

\fill[white] (0,4,0) -- (4,4,0) -- (4,4,4) -- (0,4,4) -- cycle;
\draw (0,4,0) -- (4,4,0) -- (4,4,4) -- (0,4,4) -- cycle;
\fill[white] (4,4,0) -- (4,0,0) -- (4,0,4) -- (4,4,4) -- cycle;
\draw (4,4,0) -- (4,0,0) -- (4,0,4) -- (4,4,4) -- cycle;
\fill[white] (0,0,4) -- (4,0,4) -- (4,1,4) -- (0,1,4) -- cycle;
\fill[white] (4,0,4) -- (4,4,4) -- (3,4,4) -- (3,0,4) -- cycle;
\fill[white] (0,4,4) -- (0,3,4) -- (4,3,4) -- (4,4,4) -- cycle;
\draw (0,0,4) -- (4,0,4) -- (4,4,4) -- (0,4,4) -- cycle;

\draw[densely dotted] (0,0,4) -- (0,0,0) -- (4,0,0);
\draw[densely dotted] (0,0,0) -- (0,4,0);

\draw (1,1,4) rectangle (3,3,4);

\foreach \i in {0,...,4} {
  \foreach \j in {0,...,4} {
    \node[vertex] at (\i,4,\j) {};
    \node[vertex] at (4,\i,\j) {};
    }
  }
\foreach \y in {0,...,3} {
  \foreach \z in {0,1,3} {
    \node[vertex] at (\y,\z,4) {};
    }
  }
\foreach \y in {0,1,3} {
  \node[vertex] at (\y,2,4) {};
  }

\end{tikzpicture}
\]
\caption{A schematic of the image used in Example \ref{pi1trivial-with-loophole}. (Dots have been omitted for points on 3 sides of $X$.) The loop circling the ``hole'' on the front face is not contractible by a loop-preserving homotopy, but a trivial extension is pointed contractible.\label{loopholefig}}
\end{figure}
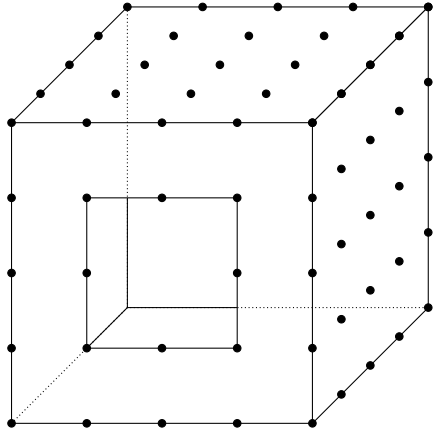

The converses of Theorem \ref{noloopholepi1} and Lemma~\ref{tricky} are not true. The following example shows an image with a loophole and trivial fundamental group.
\begin{exl}\label{pi1trivial-with-loophole}
Let $(X,6) \subset (\Z^3,6)$ be given by  $X = \delta([0,4]^3_\Z) \setminus \{(4,2,2)\}$. This image $X$ is analogous to $MSS_6$, but larger, with the center of one ``side'' deleted. A schematic of this image is shown in Figure \ref{loopholefig}. Let $f$ be the 8 point loop in $X$ which circles the deleted point $(4,2,2)$. 

By Theorem 3.12 of \cite{Staecker-etal}, the only loops homotopic to $f$ by loop-preserving homotopies are rotations of $f$. Thus $f$ does not contract by a loop-preserving homotopy, and so $X$ has a loophole. But simple modifications to the arguments used in Section \ref{han-3-3} will show that $\Pi_1^6(X,x_0)$ is trivial for any $x_0 \in X$. $\Box$
\end{exl}

The no loophole condition and the no hole condition are closely related, but not equivalent. Under a the same finiteness condition used above, ``no loophole'' is weaker than ``no hole'':
\begin{prop}\label{ccnohole}
Let $X$ be a digital image such that each component of $X$ is finite. If $X$ has no hole, then $X$ has no loophole. 
\end{prop}
\begin{proof}
Let $f:[0,m]_\Z \to X$ be a loop in $X$, and we will show that it is nullhomotopic by a loop-preserving homotopy. Let $A\subset X$ be the component of $X$ containing the points of the path $f$. Since $X$ has no hole and the components of $X$ are finite, Proposition \ref{noHoles-Contract} shows that $A$ is contractible. Let $G:A\times[0,k]_\Z \to X$ be a contraction of $A$, say $G(x,k)=a_0$ for all $x\in A$. 

Then define $H:[0,m]_\Z \times [0,k]_\Z \to X$ as $H(t,s) = G(f(t),s)$. Being a composition of continuous functions, $H$ has the necessary continuity properties to be a homotopy from $f$ to $\bar a_0$, the constant path at $a_0$. Furthermore $H$ is loop preserving since, for any $s$ we have:
\[ H(0,s) = G(f(0),s) = G(f(m),s) = H(m,s) \]
and thus each stage of $H$ is a loop. Thus $f$ is nullhomotopic by a loop-preserving homotopy as desired.
\end{proof}

The converse to Proposition \ref{ccnohole} is false, as shown by the following example.

\begin{exl}
\label{loopholeNoHole}
$(MSS_6,6)$ has no loopholes, but has a hole.
\end{exl}

\begin{proof}
The proof of Proposition~\ref{correct-3-3-1} is easily modified to show that $(MSS_6,6)$ has no loophole.

$MSS_6$ is finite, connected, and not contractible~\cite{Boxer94}. It follows from Proposition~\ref{noHoles-Contract} that $MSS_6$ has a hole.
\end{proof}

It is claimed, as Theorem~3.6
of~\cite{Han07}, that if $X$ and $Y$ are
digital surfaces in $\Z^3$ with no
$k$-holes, $k \in \{18,26\}$, then
$\Pi_1^k(X \sharp Y)$ is a trivial group. The
argument given 
depends on Theorem~3.5
of~\cite{Han07}, the flaws in which are
discussed above. Although our Theorem~\ref{noholepi1}
could be used to overcome this deficiency,
the argument for Theorem~3.6
of~\cite{Han07} also claims without proof or citation that 
$X \sharp Y$ has no $k$-holes. We 
neither have a proof nor a counterexample for this assertion
at the current writing. Thus, Theorem~3.6
of~\cite{Han07} must be regarded as unproven, and we state as open questions:
\begin{question}
If $X$ and $Y$ are digital surfaces in $\Z^3$ with no $k$-holes, is $\Pi_1^k(X\sharp Y)$ trivial?
\end{question}

\begin{question}
If $X$ and $Y$ are digital surfaces in $\Z^3$ with no $k$-holes, does $X\sharp Y$ have no
$k$-holes?
\end{question}

\begin{question}
If $X$ and $Y$ are digital surfaces in $\Z^3$ with no $k$-loopholes, does $X\sharp Y$ have no
$k$-loopholes?
\end{question}

\section{Euler characteristic}
In this section, we correct and extend several statements that
appear in Section~5 of~\cite{Han07} 
concerning the Euler characteristic
$\chi(X)$ of a digital image $X$. Some of the errors in \cite{Han07}
were previously noted in~\cite{BKO}; they are recalled 
here for completeness.

A digital image $X$ can be considered to be
a graph. When $X$ is finite, let $V=V(X)$ be the number
of vertices, i.e., the number of distinct
points of $X$; let $E=E(X)$ be the number of
distinct edges of $X$, where an edge
is given by each adjacent pair of points;
and let $F=F(X)$ be the number of distinct faces, 
where a face is an unordered
triple of distinct vertices each pair of
which is adjacent. More generally,
a {\em $k$-simplex in $X$ of dimension $d$} is a set of $d+1$ distinct members of $X$, each pair of which is $k$-adjacent.

The definition of the Euler characteristic in~\cite{Han06} is
\[ \chi(X)=V-E+F.\]
This definition is satisfactory
if $X$ has no simplices of dimension greater than 2. However,
the latter assumption is not always
correct, even for digital surfaces; e.g., $MSC_8^*$ has
3-simplices.
Thus, a better definition of the
Euler characteristic is that of~\cite{BKO}:
\[ \chi(X)= \chi(X,k)=\sum_{q=0}^m (-1)^q \alpha_q, \]
where $m$ is the largest integer $d$ such that $(X,k)$ has a simplex of dimension~$d$ and $\alpha_q$ is 
the number of distinct $q$-dimensional $k$-simplices in $X$.

At statement~(5.1) of~\cite{Han07}, 
it is inferred that, using 18-adjacency in $\Z^3$,
\[ V(MSS_{18}) =10,~~E(MSS_{18}) = 20,
   ~~F(MSS_{18}) = 12,
\]
and therefore that $\chi(MSS_{18})=2$. In fact, one sees easily (see Figure~\ref{3surfs}) that
$F(MSS_{18})=8$, namely, the faces are
\begin{align*}
\langle c_0,c_1,c_9\rangle , \langle c_0,c_1,c_6\rangle , \langle c_0,c_5,c_6\rangle , \langle c_0,c_5,c_9\rangle , 
\\
\langle c_2,c_3,c_7\rangle , \langle c_2,c_3,c_8\rangle ,
\langle c_3,c_4,c_7\rangle , \langle c_3,c_4,c_8\rangle ,
\end{align*} 

and therefore, as noted in~\cite{BKO}, we have the following.

\begin{exl}
$\chi(MSS_{18})=-2$.
\end{exl}

Theorem~5.2 of~\cite{Han07} claims that for closed $k$-surfaces $X$ and $Y$,
\[\chi(X \sharp Y)=\chi(X) + \chi(Y) -2. \]
This formula is attractive because it matches the classical formula for the Euler characteristic of a connected sum of surfaces. Unfortunately we will see that the formula holds only in some cases. The argument given in \cite{Han07} makes some counting errors, and fails to count 3-simplices. A correct formula must include the Euler characteristic of $A_k$.

\begin{lem}
\label{chiSharpFormula}
For closed digital surfaces $X$ and $Y$,
we have
\[ \chi(X\sharp Y) = \chi(X) + \chi(Y) - 2\chi(A_k). \]
\end{lem}
\begin{proof}
Recall that $\delta(A_k)$ denotes the boundary of $A_k$. The construction of $X\sharp Y$ can be thought of as deleting $A_k$ from each of $X$ and $Y$, and then reinserting only one copy of $\delta(A_k)$. Because $X$ and $Y$ are digital surfaces with $A_k$ embedded inside, no simplex of $X$ or of $Y$ has vertices in both the interior and exterior of $A_k$. Thus when we delete $A_k$ from each of $X$ and $Y$, this deletes only simplices of $A_k$, and when we reinsert $\delta(A_k)$, this inserts only simplices of $\delta(A_k)$. Thus in each dimension $q$ we have:
\[ \alpha_q(X\sharp Y) = \alpha_q(X) + \alpha_q(Y) - 2 \alpha_q(A_k) + \alpha_q(\delta(A_k)), \]
where $\alpha_q$ is the number of $q$-simplices. Taking the alternating sum above we obtain:
\[ \chi(X\sharp Y) = \chi(X) + \chi(Y) - 2\chi(A_k) + \chi(\delta(A_k)). \]

It remains to show that $\chi(\delta(A_k)) = 0$. We can check easily in Figure \ref{hanDisks} that in each possible case for $A_k$, the boundary $\delta(A_k)$ is a simple cycle of points. Thus $\chi(\delta(A_k)) = 0$ as desired.
\end{proof}

The next result was obtained for
digital surfaces in ~\cite{Han07} 
and generalized in~\cite{BKO}.

\begin{prop}
\label{isoChi}
Isomorphic digital images have 
the same Euler characteristic. $\Box$
\end{prop}

\begin{exl}
\label{chi-disks}
We have the following.
\begin{itemize}
\item $\chi(MSC_8^*) = 1$.
\item $\chi(MSC_8'^*) = 1$.
\item $\chi(MSC_4^*) = -3$.
\end{itemize}
\end{exl}

\begin{proof} See Figure~\ref{hanDisks}. For
$(MSC_8^*,8)$, we see there are
8 vertices, 17 edges, 12
faces, 2 3-simplices, and no simplices of dimension greater than 3, so
\[ \chi(MSC_8^*)=8-17+12-2=1.
\]

For $(MSC_8'^*,8)$, we see there are
5 vertices, 8 edges, 4 faces, and 
no simplices of dimension greater than 2, so
\[ \chi(MSC_8'^*)=5-8+4=1.
\]

For $(MSC_4^*,4)$, we see there are
9 vertices, 12 edges, and 0
simplices of dimension greater than 1, so
\[ \chi(MSC_4^*)=9-12=-3.
\]
\end{proof}

The computations above immediately give a corrected version of Theorem 5.2 of \cite{Han07}:

\begin{thm}
\label{chiSharp}
For closed digital surfaces $X$ and $Y$,
\[ \chi(X \sharp Y)=
   \left \{ \begin{array}{ll}
   \chi(X)+\chi(Y)-2 & \mbox{ \rm{if} } A_k \approx_{(k,8)} MSC_8^*; \\
   \chi(X)+\chi(Y)-2 & \mbox{ \rm{if} } A_k \approx_{(k,8)} MSC_8'^*; \\
   \chi(X)+\chi(Y)+6 & \mbox{ \rm{if} } A_k \approx_{(k,4)} MSC_4^*.
   \end{array} \right .
\]
\end{thm}

\begin{proof} The assertion follows
from Lemma~\ref{chiSharpFormula} and
Example~\ref{chi-disks}.
\end{proof}

Example~5.3 of~\cite{Han07} claims incorrectly that
\[\chi(MSS_{18} \sharp MSS_{18})=\chi(MSS_{18})= 2\]
and that
\[ \chi(MSS_{18}' \sharp MSS_{18}) = 
\chi(MSS_{18}') = 2. \]
Examples~\ref{chiEx} and~\ref{chiSharp} below correct
these errors.

\begin{exl}
\rm{\cite{BKO}}
\label{chiEx}
\begin{itemize}
\item $\chi(MSS_{18} \sharp MSS_{18})=-6$.
\item $\chi(MSS_{18})=-2$.
\item $\chi(MSS_{18}') = 2$. $\Box$
\end{itemize}
\end{exl}

\begin{exl}
\label{chiSharp18'18}
$\chi(MSS_{18}' \sharp MSS_{18})=
  -2$.
\end{exl}

\begin{proof}
Using $A_{18} \approx_{(18,8)} MSC_8'^*$, 
it is easily observed that 
$MSS_{18}' \sharp MSS_{18}$ and
$MSS_{18}$ are 18-isomorphic.
From Proposition~\ref{isoChi},
$\chi(MSS_{18}' \sharp MSS_{18})=
\chi(MSS_{18})$. The assertion
follows from Example~\ref{chiEx}.
\end{proof}

\section{Further remarks}
We have given corrections to many errors that
appear in~\cite{Han07} concerning fundamental groups
and Euler characteristics of 2-sphere-like digital
images. We have also presented some original results related to these ideas, including an example that
shows that contractibility does not imply pointed
contractibility among digital images, and our results concerning
``no loopholes."

\end{document}